\documentclass[a4paper]{article}

\usepackage[english]{babel}
\usepackage[utf8]{inputenc}
\usepackage{amssymb}
\usepackage{graphicx}
\usepackage[colorinlistoftodos]{todonotes}
\usepackage{amsmath}
\usepackage{amsthm}

\theoremstyle{plain}
\newtheorem{theorem}{Theorem}

\newtheorem{lemma}{Lemma}

\theoremstyle{definition}
\newtheorem{definition}{Definition}

\title{A Partition Function Connected with the G\"ollnitz--Gordon Identities}

\author{Nicolas Allen Smoot \\ Georgia Southern University}

\date{}

\begin{document}
\maketitle

\begin{abstract}
We use the celebrated circle method of Hardy and Ramanujan to develop convergent formul\ae\ for counting a restricted class of partitions that arise from the G\"ollnitz--Gordon identities.
\end{abstract}

\section{Introduction}

The purpose of this article is to illustrate a beautiful application of the tools of complex analysis to a discrete subject: the theory of addition over the integers, also known as partition theory.

A partition of a positive integer $n$ is simply an expression of $n$ as a sum of other positive integers.  For example, taking the number $5$, we find $7$ different partitions: $5$, $4+1$, $3+2$, $3+1+1$, $2+2+1$, $2+1+1+1$, and $1+1+1+1+1$.

The number of partitions of $n$ is denoted by $p(n)$, and is often called the partition function.  In our example, we have $p(5)=7$.

While partitions have been studied since the time of Euler \cite{Hardy2}, very little was known about the partition function itself before the twentieth century.  Indeed, at the end of the nineteenth century, attempts to study the behavior of the prime counting function \cite{Stein} had led to a general sense of pessimism in number theory \cite{Hardy1}; it was expected that any careful analysis of $p(n)$ would produce an asymptotic formula that was approximate at best, and certainly not useful for direct computation.

It was not until 1918 that Hardy and Ramanujan developed the techniques to conduct a detailed study of $p(n)$ \cite{Hardy1}.  The results of their work were astonishing: not only were they capable of achieving a formula that could give the exact value of $p(n)$ with relative efficiency, but the formula itself is an utterly bizarre object, as an infinite series containing Bessel functions, coprime sums over roots of unity, and $\pi$---analytic entities that seem wholly irrelevant to the question of simple addition over the natural numbers.

The techniques that Hardy and Ramanujan had developed are embodied in what is now known as the circle method.  This method has since become one of the most basic tools in analytic number theory \cite{Rademacher3}, \cite{Vaughan}.

Notably, the circle method has continued to contribute to the theory of partitions and $q$-series.  Hardy and Ramanujan's formula was carefully refined by Rademacher, first in 1936 \cite{Rademacher1} to make their formula for $p(n)$ convergent, and again in 1943 \cite{Rademacher2} as an adjustment of the method itself.  Soon thereafter, it was realized that the techniques embodying the circle method could be used to develop formul\ae\ for a variety of more restricted partition functions (two notable examples are \cite{Niven} and \cite{Lehner}).

We are interested here in one such partition function, associated with the G\"ollnitz--Gordon identities \cite{Gollnitz}, \cite{Gordon}, which we provide here for reference:

\begin{theorem}[G\"ollnitz--Gordon Identities]

Fix $a$ to be either $1$ or $3$.  Given an integer $n$, the number of partitions of $n$ in which parts are congruent to $4, \pm a \pmod 8$, is equal to the number of partitions of $n$ in which parts are non-repeating and non-consecutive, with any two even parts differing by at least $4$, and with all parts $\ge a$. 

\end{theorem}

Each identity---one for either value of $a$---equates the sizes of two different classes of partitions of $n$, while not actually indicating the class size itself.  We will use Hardy and Ramanujan's method, together with Rademacher's refinements, to formulate a convergent expression for the number of partitions associated with these identities.

\begin{definition}
Fix $a$ at either $1$ or $3$.  A G\"ollnitz--Gordon partition of type $a$ is composed of parts of the form $4,\pm a \pmod 8$.  The generating function for such partitions is expressed as $F_a(q)$, and the actual number of such partitions of $n$ is given as $g_a(n)$.
\end{definition}

We seek a formula for $g_a(n)$.  The author wishes to note his deep appreciation for the guidance and encouragement of Professor Andrew Sills, who first suggested this problem.

In keeping with the theory of $q$-series, we have

\begin{align}
F_a(q) =& \sum\limits_{k=0}^{\infty} g_a(k) q^k\\ =& \prod\limits_{m=0}^{\infty} (1 - q^{8 m + a})^{-1} (1 - q^{8 m + 4})^{-1} (1 - q^{8 m + 8 - a})^{-1}\\
=& \frac{1}{(q^a;q^8)_{\infty} (q^4;q^8)_{\infty} (q^{8 - a};q^8)_{\infty}},\label{generatingfunction1}
\end{align} with

\begin{equation}
(a;q)_{\infty} = \prod_{j=0}^{\infty}\left(1 - aq^j\right).
\end{equation}

Cauchy's residue theorem \cite[Chapter 3]{Stein} gives us a means of calculating---at least in principle---the value of $g_a(n)$.  Dividing $F_a(q)$ by $q^{n+1}$, we find that $g_a(n)$ is the coefficient of $q^{-1}$, and is therefore the residue of $F_a(q)/q^{n+1}$.

\begin{theorem}

\begin{equation}
g_a(n) = \frac{1}{2\pi i} \oint\limits_{\mathcal{C}} \frac{F_a(q)}{q^{n+1}} dq\label{theorcal},
\end{equation} for $\mathcal{C}$ some curve inside the unit circle of the $q$-plane, encompassing $q=0$.

\end{theorem}

We must choose an appropriate contour for $\mathcal{C}$.  We then study $F_a(q)$ itself, including some of its useful transformation properties.  Next, we will employ the circle method in reducing our integral (\ref{theorcal}) to something far more accessible to integration.  We finish our integration using the theory of Bessel functions.

\section{Rademacher's Contour}

Casual inspection of (\ref{generatingfunction1}) suggests that $F_a(q)$ has important structure near the roots of unity of the unit circle.  We will construct a contour that remains inside the unit circle, but approaches the roots of unity $e^{2\pi i h/k}$ in a controlled way.  This contour was first used by Rademacher \cite{Rademacher2}.

\begin{definition}
\normalfont
For a given $h/k \in \mathcal{F}_N$, define the Ford circle $C(h,k)$ as the curve given by 

\begin{equation}
\left| \tau - \bigg(\frac{h}{k} + \frac{i}{2k^2}\bigg) \right| = \frac{1}{2k^2}\label{ford1}.
\end{equation}

\end{definition}

Given the set of Ford circles corresponding to the Farey sequence of degree $N$, let $\gamma (h,k)$ be defined as the upper arc of $C(h,k)$ from 

\begin{equation*}
\tau_I(h,k) = \frac{h}{k} - \frac{k_p}{k (k^2 + k_p^2)} + \frac{1}{k^2 + k_p^2} i
\end{equation*}

to

\begin{equation*}
\tau_T(h,k) = \frac{h}{k} + \frac{k_s}{k (k^2 + k_s^2)} + \frac{1}{k^2 + k_s^2} i,
\end{equation*} with $h_p/k_p$ and $h_s/k_s$ the immediate predecessor and successor (respectively) of $h/k \in \mathcal{F}_N$ (let $0_p/1_p = (N-1)/N$; similarly, let $(N-1)_s/N_s = 0/1$).

\begin{definition}
\normalfont
The Rademacher path of order $N$, $P(N)$, is the union of all upper arcs $\gamma (h,k)$ from $\tau = i$ to $\tau = i + 1$:

\begin{equation}
P(N) = \bigcup_{h/k \in \mathcal{F}_N} \gamma (h,k).
\end{equation}

\end{definition}

We give an illustration of $P(3)$ in Figure 1.

It may be easily demonstrated that consecutive Ford circles corresponding to $\mathcal{F}_N$ are tangent to one another, so that $P(N)$ is a connected curve.  Moreover, for $\tau$ in the upper arc $\gamma (h,k)$, $\Im(\tau)>0$; therefore, $\gamma(h,k)$ lies entirely in $\mathbb{H}$ for every $h/k \in \mathcal{F}_N$.  Therefore, $P(N)$ is a connected curve that lies entirely in $\mathbb{H}$.

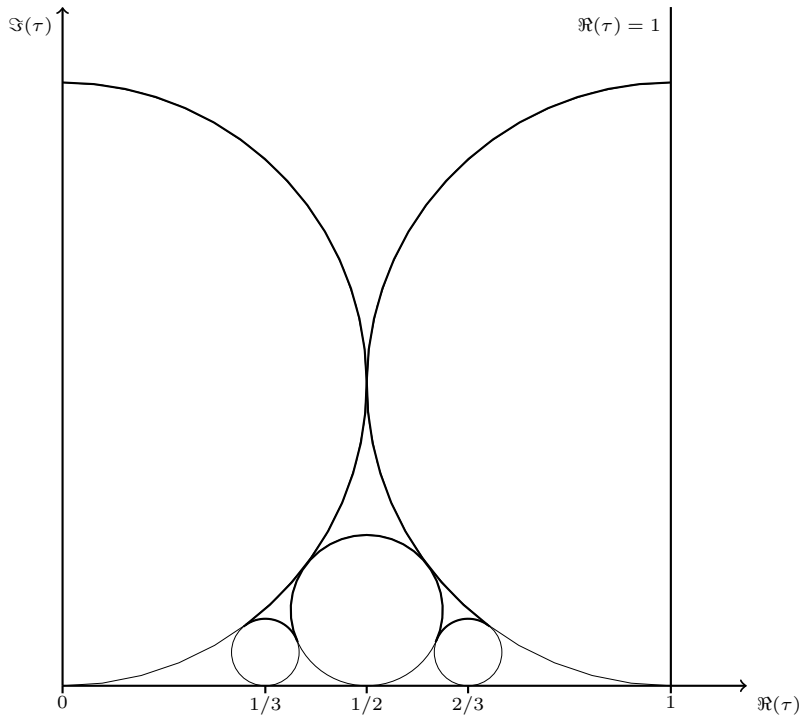
\begin{figure}[h]
\centering
\begin{tikzpicture}
    \begin{scope}[thick,font=\scriptsize][set layers]
    \draw [->] (-4,-4) -- (5,-4) node [below right]  {$\Re(\tau)$};
    \draw [->] (-4,-4) -- (-4,5) node [below left] {$\Im(\tau)$};
\draw [-] (4,-4) -- (4,5) node [below left] {$\Re(\tau)=1$};
\draw [-] (-4,-4.1) -- (-4,-4) node [below] {$0$};
\draw [-] (0,-4.1) -- (0,-4) node [below] {$1/2$};
\draw [-] (4,-4.1) -- (4,-4) node [below] {$1$};
\draw [-] (-4/3,-4.1) -- (-4/3,-4) node [below] {$1/3$};
\draw [-] (4/3,-4.1) -- (4/3,-4) node [below] {$2/3$};
\end{scope}
\draw [very thin,domain=-90:90] plot ({-4+4*cos(\x)}, {4*sin(\x)});
\draw [very thin,domain=90:270] plot ({4+4*cos(\x)}, {4*sin(\x)});
\draw [very thin] (0,-3) circle (1);
\draw [very thin] (-4/3,-32/9) circle (4/9);
\draw [thick,domain=-53:90] plot ({-4+4*cos(\x)}, {4*sin(\x)});
\draw [very thin] (4/3,-32/9) circle (4/9);
\draw [thick,domain=90:233] plot ({4+4*cos(\x)}, {4*sin(\x)});
\draw [thick,domain=20:130] plot ({-4/3+4/9*cos(\x)}, {-32/9+4/9*sin(\x)});
\draw [thick,domain=50:160] plot ({4/3+4/9*cos(\x)}, {-32/9+4/9*sin(\x)});
\draw [thick,domain=-25:205] plot ({cos(\x)}, {-3+sin(\x)});
\end{tikzpicture}
\caption{Ford circles $C(h,k)$ for $h/k\in\mathcal{F}_3$, with $P(3)$ highlighted.}\label{circlef1}
\end{figure}

So if we define $q = e^{2 \pi i \tau}$, then we may define our curve $\mathcal{C}$ from (\ref{theorcal}) as the preimage of $P(N)$.  We will make one more helpful change of variables:

\begin{equation}
\tau = \frac{h}{k} + \frac{iz}{k},
\end{equation} with $\Re(z) > 0$.  This change maps $C(h,k)$ (with $\gamma(h,k)$) to the circle

\begin{equation}
K_k^{(-)}: \left|z - \frac{1}{2k}\right| = \frac{1}{2k}\label{ford2}.
\end{equation}

Notice that the initial and terminal points of $\gamma(h,k)$ are mapped to $z_I(h,k)$ and $z_T(h,k)$ by the following:

\begin{equation}
\tau_I(h,k) \mapsto z_I(h,k) = \frac{k}{k^2+k_p^2} + \frac{k_p}{k^2+k_p^2}i,
\end{equation}

\begin{equation}
\tau_T(h,k) \mapsto z_T(h,k) = \frac{k}{k^2+k_s^2} - \frac{k_s}{k^2+k_s^2}i.
\end{equation}  We finish this section by referencing an important lemma, which can be proved quickly from the properties of the Farey fractions \cite{Hardy2}.

\begin{lemma}

Let $N\in\mathbb{N}$ be given, with $h/k \in \mathcal{F}_N$.  Let $z_I(h,k)$, $z_T(h,k)$ be the images of $\tau_I(h,k)$, $\tau_T(h,k),$ respectively, from $C(h,k)$ to $K_k^{(-)}$.  Then for any $z$ on the chord connecting $z_I(h,k)$ to $z_T(h,k)$, we have 

\begin{equation}
|z| = O\left(N^{-1}\right).
\end{equation}

\end{lemma}

\section{Transformation Equations}

We begin by expressing $F_a(q)$ in terms of automorphic forms---in particular, as a quotient of eta functions by a theta function.  Let $q=e^{2\pi i\tau}$, with $\tau$ a complex variable, $\Im(\tau) > 0$.

Recall that Ramanujan's theta function \cite[Chapter 1]{Berndt} has the following product expansion:

\begin{equation}
f(-q^{\alpha},-q^{\beta}) = (q^{\alpha};q^{\alpha + \beta})_{\infty} (q^{\beta};q^{\alpha + \beta})_{\infty} (q^{\alpha + \beta};q^{\alpha + \beta})_{\infty}\label{ramanujanthetaformula}.
\end{equation}

Moreover, Ramanujan's theta function is related to the standard theta function $\vartheta_1$ by the following, which can be verified by the series representations of both functions \cite[Chapter 10]{Rademacher3}:

\begin{equation}
f(-q^{\alpha},-q^{\beta}) = -i e^{\pi i \tau (3\alpha - \beta)/4} \vartheta_1 (\alpha \tau | (\alpha + \beta)\tau).
\end{equation}

We then have

\begin{align}
F_a(q) =& \frac{(q^8;q^8)_{\infty}^2}{(q^4;q^4)_{\infty} f(-q^a,-q^{8-a})}\\ =& i \exp(\pi i \tau (2 - a)) \frac{(q^8;q^8)_{\infty}^2}{(q^4;q^4)_{\infty} \vartheta_1(a\tau | 8\tau)}.
\end{align}

Since $(q^{\alpha};q^{\alpha})_{\infty} = e^{-\alpha\pi i\tau/12} \eta (\alpha\tau)$, we can rewrite the remaining $q$-Pochhammer symbols in terms of eta functions in the following way:

\begin{equation}
F_a(q) = i \exp(\pi i\tau (1-a)) \frac{\eta(8\tau)^2}{\eta(4\tau) \vartheta_1(a\tau|8\tau)}.
\end{equation}

We will now study the behavior of $F_a(q)$ near the arbitrary singularity $e^{2\pi i h/k}$, with $0\le h < k$, and $(h,k)=1$.  To do this, we will divide our work into four cases, depending on the divisibility properties of $k$ with respect to $8$, and then take advantage of the modular symmetries of the $\eta$ and $\vartheta_1$ functions.

\subsection{$GCD(k,8) = 8$}

The simplest transformation formula relevant to our problem occurs for $(k,8) = 8$.  Let $H_8$ be defined as the negative inverse of $h$ modulo $16k$:

\begin{equation}
hH_8 \equiv -1 \pmod {16k}\label{inversecase8}.
\end{equation}

Notice that since $8|k$ by hypothesis, and $(h,k) = 1$, therefore $(h,16k) = (h,k) = 1$, so that $H_8$ exists.  Then the following  are elements of $SL(2,\mathbb{Z})$:

\begin{align}
\begin{pmatrix} h & -\frac{8}{k} (hH_8 + 1) \\ \frac{k}{8} & -H_8 \end{pmatrix}\label{matrix8c8},
\end{align}

\begin{align}
\begin{pmatrix} h & -\frac{4}{k} (hH_8 + 1) \\ \frac{k}{4} & -H_8 \end{pmatrix}\label{matrix8c4},
\end{align}

We will allow 

\begin{equation}
\tau' = \frac{H_8}{k} + \frac{iz^{-1}}{k}.
\end{equation}

Applying (\ref{matrix8c8}) as a modular transformation to $8\tau'$, we have

\begin{equation*}
\frac{8h\tau' - \frac{8}{k}(hH_8 + 1)}{8\frac{k}{8}\tau' - H_8} = 8\tau.
\end{equation*}

Similarly, applying (\ref{matrix8c4}) to $4\tau'$, we get $4\tau$.

Therefore, we will transform $\eta(8\tau)$ to $\eta(8\tau')$, using (\ref{matrix8c8}).  Similarly, we transform $\eta(4\tau)$ to $\eta(4\tau')$ using (\ref{matrix8c4}).  

Invoking these transformations, we must contend with the roots of unity associated with the $\eta$ and $\vartheta_1$ functions.  As a shorthand, we will refer to the roots of unity as the following:

\begin{equation}
\epsilon(8,8) = \epsilon\left(h,-\frac{8}{k}(hH_8 + 1),\frac{k}{8},-H_8\right)\label{appbeta81},
\end{equation}

\begin{equation}
\epsilon(8,4) = \epsilon\left(h,-\frac{4}{k}(hH_8 + 1),\frac{k}{4},-H_8\right)\label{appbeta82},
\end{equation} where $\epsilon(a,b,c,d)$ is the root of unity given by 

\begin{equation}
\epsilon(a,b,c,d) =
  \begin{cases}
  \big(\frac{d}{c}\big) i^{(1 - c)/2} \exp\left(\frac{\pi i}{12}(bd(1-c^2)+c(a+d))\right), &\quad 2 \nmid c \\
    \big(\frac{c}{d}\big) \exp\left(\frac{\pi i d}{4} + \frac{\pi i}{12}(ac(1-d^2)+d(b-c))\right), &\quad 2 \nmid d \\
  \end{cases},\label{etaroot}\end{equation} and $\big(\frac{m}{n}\big)$ is the Legendre--Jacobi character.  See \cite[Chapter 9]{Rademacher3}.

Invoking the functional equation for $\eta$ \cite[Chapter 9]{Rademacher3}, it follows that

\begin{equation}
\frac{\eta(8\tau)^2}{\eta(4\tau)} = \frac{1}{z^{1/2}} \frac{\epsilon(8,8)^2}{\epsilon(8,4)} \frac{\eta(8\tau')^2}{\eta(4\tau')}\label{etas8}.
\end{equation}

Handling $\vartheta_1$ turns out to be more difficult, due to the presence of a second complex variable.  We will mimic our work with $\eta(8\tau)$, using (\ref{matrix8c8}), and setting 

\begin{equation}
v= a\tau i z^{-1} = \frac{a(hiz^{-1} - 1)}{k}\label{v8}.
\end{equation}  The functional equation for $\vartheta_1$ \cite[Chapter 10]{Rademacher3} gives us

\begin{equation}
\vartheta_1(a\tau|8\tau) = \vartheta_1\bigg( \frac{v}{iz^{-1}} \bigg| 8\tau \bigg) = -i \epsilon(8,8)^3 \frac{1}{z^{1/2}} e^{z \pi a^2(hiz^{-1} - 1)^2/8k} \vartheta_1\left(v|8\tau' \right)\label{theta8s1}.
\end{equation}

Recall that $hH_8 \equiv -1 \pmod {16k}$.  We can therefore write 

\begin{equation}
-1 = hH_8 + 16kM,
\end{equation} with $M\in\mathbb{Z}$.  We then have 

\begin{equation}
v = ah\tau' + 16aM\label{veven}.
\end{equation}  If we also take advantage of the fact that $\vartheta_1(v+1|\tau) = -\vartheta_1(v|\tau)$ \cite[Chapter 10]{Rademacher3}, then we have

\begin{equation}
\vartheta_1\left(v|8\tau' \right) = \vartheta_1\left(ah\tau' + 16aM| 8\tau' \right) = \vartheta_1\left(ah\tau' | 8\tau'\right)\label{theta8s2}.
\end{equation}

Again considering that $(k,8) = 8$ and $(h,k) = 1$, and $a=1,3$, we also have $ah \equiv 1,3,5,7 \pmod 8$.  We therefore write

\begin{equation}
\vartheta_1(ah\tau' | 8\tau') = \vartheta_1(b\tau' + 8N\tau' | 8\tau')\label{appbB8},
\end{equation} with $b$ the least positive residue of $ah$ modulo 8.

We now make use of the fact that for $N\in\mathbb{N}$,

\begin{equation}
\vartheta_1(v+N\tau|\tau) = (-1)^N\exp\left(-\pi i N(2v + N\tau)\right)\vartheta_1(v|\tau)\label{addNvarfortheta}
\end{equation} \cite[Chapter 10]{Rademacher3}, so that

\begin{equation}
\vartheta_1(ah\tau' | 8\tau') = (-1)^N \exp(-\pi i N(2b\tau' + 8N\tau')) \vartheta_1(b\tau'|8\tau')\label{theta8s3}.
\end{equation}

Combining (\ref{theta8s1}), (\ref{theta8s2}), (\ref{theta8s3}), and inverting, we have the following:

\begin{equation}
\frac{1}{\vartheta_1(a\tau|8\tau)} =  i \frac{(-1)^N}{\epsilon(8,8)^3} {z^{1/2}} e^{-z \pi a^2(hiz^{-1} - 1)^2/8k} \frac{\exp(\pi i N(2b\tau' + 8N\tau'))}{\vartheta_1(b\tau'|8\tau')}\label{theta8s4}.
\end{equation}

We now have sufficient information, in (\ref{etas8}), (\ref{theta8s4}), to reassemble the transformed generating function.

\begin{align}
F_a(q) =& i\exp(\pi i \tau(1 - a)) \frac{1}{z^{1/2}} \frac{\epsilon(8,8)^2}{\epsilon(8,4)} \frac{\eta(8\tau')^2}{\eta(4\tau')} \frac{1}{\vartheta_1(a\tau|8\tau)}\\
&\times \exp(\pi i N(2b\tau' + 8N\tau'))\frac{\eta(8\tau')^2}{\eta(4\tau')\vartheta_1(b\tau'|8\tau')}\label{gen8s1}.
\end{align}

Here $b=1,3,5,7$.  However, noting from (\ref{ramanujanthetaformula}) that 

\begin{equation}
f(-q^{\alpha},-q^{\beta}) = f(-q^{\beta},-q^{\alpha}),
\end{equation} we may define $F_5(q)=F_3(q),$ $F_7(q)=F_1(q)$.  We therefore have

\begin{multline}
F_a(q) = \frac{i(-1)^{N}}{\epsilon(8,8)\epsilon(8,4)} \times\exp\bigg(\pi i \tau(1 - a) -z \pi a^2(hiz^{-1} - 1)^2/8k\\ + \pi i N(2b\tau' + 8N\tau')+\pi i\tau'(a - 1)\bigg)F_b(y)\label{gen8s2},
\end{multline} with $y=\exp( 2\pi i \tau')$.

Remembering that

\begin{equation}
N = \left\lfloor \frac{ah}{8} \right\rfloor = \frac{ah - b}{8}\label{appbN8},
\end{equation} and that 

\begin{equation}
a^2 - 4a + 3 = (a - 1)(a - 3) = 0,
\end{equation} we may collect and reorganize the coefficients of 1, $z$, and $1/z$ in the exponential of (\ref{gen8s2}).  Doing so gives the following transformation formula:

\begin{equation}
F_a(q) = \omega_{a,8} (h,k) \exp\left(\frac{\pi}{8k} \left(\frac{(b-4)^2-8}{z} + z(4a-5)\right)\right) F_b (y)\label{gen8final},
\end{equation} where

\begin{equation}
\omega_{a,8} (h,k) = \frac{i (-1)^{\lfloor \frac{ah}{8} \rfloor}}{\epsilon (8,8) \epsilon (8,4)} \exp\left(\frac{\pi i}{8k} (h (5 - 4a) - H_8((b-4)^2-8))\right)\label{rootun8}.
\end{equation}

We note that we can extend this result to prove the modularity of $F_a(q)$ relative to a certain subgroup of $SL(2,\mathbb{Z})$.  We do not give the proof here.

\subsection{$GCD(k,8) < 8$}

The result of the Section 3.1 suggests that $F_a(q)$ is modular, at least with respect to a subgroup of the modular group.  While such a property does not carry over exactly to the remaining 3 cases, it is only necessary to show that $F_a(q) = f(z) \Psi(y)$, with $\Psi(q)$ a suitable quotient of $q$-series.

For each case $(k,8)=d$, we will define

\begin{equation}
\tau' = \frac{H_d}{k} + \frac{diz^{-1}}{8k},
\end{equation} where

\begin{equation}
\frac{8hH_d}{d} \equiv -1 \pmod {k/d}\label{inverseDefn},
\end{equation}  and

\begin{equation}
y = e^{2\pi i\tau'}.
\end{equation}  We consider the following matrices, which can easily be shown to be in $SL(2,\mathbb{Z})$:

\begin{align}
\begin{pmatrix}8h/d & -\frac{d}{k} (8hH_d/d + 1) \\ k/d & -H_d \end{pmatrix}\label{matrix4c8},
\end{align}

\begin{align}
\begin{pmatrix} 4h/d & -\frac{d}{k} (8hH_d/d + 1) \\ k/d & -2H_d \end{pmatrix}\label{matrix4c4}.
\end{align}  We also define

\begin{equation}
\epsilon(d,8) = \epsilon\left(8h/d,-\frac{d}{k}(8hH_d/d + 1),\frac{k}{d},-H_d\right),
\end{equation}

\begin{equation}
\epsilon(d,4) = \epsilon\left(4h/d,-\frac{d}{k}(8hH_d/d + 1),\frac{k}{d},-2H_d\right),
\end{equation} with $\epsilon(a,b,c,d)$ defined by (\ref{etaroot}).  Remembering (\ref{inverseDefn}), we also let

\begin{equation}
v = \frac{da(hiz^{-1} - 1)}{8k} = ah\tau' + \frac{aM}{8},
\end{equation} with $M\in\mathbb{Z}$.  Finally, we write

\begin{equation}
\rho_{a,d} = \exp\left( \frac{\pi i a d}{4k} \left( \frac{8hH_d}{d} + 1 \right) \right).
\end{equation}

\subsubsection{$GCD(k,8) = 4$}

With $d=4$, we apply (\ref{matrix4c8}) to $4\tau'$, and (\ref{matrix4c4}) to $8\tau'$, so that we have

\begin{equation}
\frac{\eta(8\tau)^2}{\eta(4\tau)} = \frac{1}{2z^{1/2}} \frac{\epsilon(4,8)^2}{\epsilon(4,4)} \frac{\eta(4\tau')^2}{\eta(8\tau')}\label{etas4}.
\end{equation}

As with the case of $d=8$, $\vartheta_1$ requires the most work by far.  The initial transformation through (\ref{matrix4c8}) gives us

\begin{equation}
\vartheta_1(a\tau|8\tau) = -i \epsilon(4,8)^3 \frac{1}{(2z)^{1/2}} e^{z \pi a^2(hiz^{-1} - 1)^2/8k} \vartheta_1\left(v|4\tau'\right)\label{theta4s1}.
\end{equation}  And

\begin{equation}
\vartheta_1\left(v | 4\tau' \right) = \vartheta_1\left(ah\tau' + \frac{aM}{8} \bigg| 4\tau'\right)\label{c4adjustv}.
\end{equation}

We may now allow $b\equiv ah \pmod 4$, letting $ah = 4N + b$, so that (\ref{c4adjustv}), together with (\ref{addNvarfortheta}), gives

\begin{multline}
\vartheta_1\left(v | 4\tau'\right) =  (-1)^N \exp(-\pi i N (2\tau'(2N + b) + aM/4))\\ \times\vartheta_1\left(b\tau' + \frac{aM}{8} \bigg| 4\tau'\right)\label{theta4s2}.
\end{multline}

We now shift from $\vartheta_1$ to $\vartheta_4$ \cite[Chapter 10]{Rademacher3}:

\begin{equation*}
\vartheta_1(v|\tau) = i \exp(-\pi i\tau/4 - \pi i v) \vartheta_4(v|\tau).
\end{equation*}

Write

\begin{align}
\vartheta_1\left(b\tau' + \frac{aM}{8} \bigg| 4\tau'\right) =& \vartheta_1\left((b - 2)\tau' + \frac{aM}{8} + 2\tau' \bigg| 4\tau'\right)\\ =& i \exp(-\pi i((b - 1)\tau' + aM/8))\\ &\times \vartheta_4\left((b - 2)\tau' + \frac{aM}{8} \bigg| 4\tau'\right)\label{theta4s3}.
\end{align}

We now express $\vartheta_4$ as an infinite product \cite[Chapter 10]{Rademacher3}:

\begin{align}
\vartheta_4\left((b - 2)\tau' + \frac{aM}{8} \bigg| 4\tau'\right) &= \prod\limits_{m=1}^{\infty} (1 - y^{4 m})(1 - \rho_{a,4} y^{4 m - 4 + b})(1 - \rho_{a,4}^{-1}y^{4 m - b})\notag \\ &= (y^4;y^4)_{\infty}(\rho_{a,4} y^b;y^4)_{\infty} (\rho_{a,4}^{-1}y^{4 - b};y^4)_{\infty}\label{theta4s4}.
\end{align}

Combining (\ref{etas4}), (\ref{theta4s1}), (\ref{theta4s2}), (\ref{theta4s3}), (\ref{theta4s4}), and simplifying, we have

\begin{equation}
F_a(q) = \frac{1}{\sqrt{2}} \omega_{a,4} (h,k) \exp\left(\frac{\pi}{8k} \left(\frac{1}{z} + z(4a-5)\right)\right) \Psi_{a,4} (y)\label{gen4final},
\end{equation} with

\begin{equation}
\Psi_{a,4}(q) = \frac{(q^4;q^4)_{\infty}^2}{(q^8;q^8)_{\infty}f(-\rho_{a,4} q^b;-\rho_{a,4}^{-1}q^{4-b})}\label{psi4},
\end{equation} and

\begin{multline}
\omega_{a,4} (h,k) = \frac{i (-1)^{\lfloor \frac{ah}{4} \rfloor}}{\epsilon (4,8) \epsilon (4,4)}\\ \times \exp\left( \frac{\pi i}{4k} \left(h - H_4 - h (4a - 3) (hH_4 + 1) + a (2hH_4 + 1) (b-2)\right)\right)\label{rootun4}.
\end{multline}

\subsubsection{$GCD(k,8) = 2$}

With $d=4$, we apply (\ref{matrix4c8}) to $2\tau'$, and (\ref{matrix4c4}) to $4\tau'$, so that we have

\begin{equation}
\frac{\eta(8\tau)^2}{\eta(4\tau)} = \frac{1}{2\sqrt{2}z^{1/2}} \frac{\epsilon(2,8)^2}{\epsilon(2,4)} \frac{\eta(2\tau')^2}{\eta(4\tau')}\label{etas2}.
\end{equation}

Once again, $\vartheta_1$ requires the most work by far.  The initial transformation through (\ref{matrix4c8}) gives us

\begin{equation}
\vartheta_1(a\tau|8\tau) = -i \epsilon(2,8)^3 \frac{1}{2z^{1/2}} e^{z \pi a^2(hiz^{-1} - 1)^2/8k} \vartheta_1\left( v | 2\tau' \right)\label{theta2s1}.
\end{equation}  And

\begin{equation}
\vartheta_1\left(v | 2\tau' \right) = \vartheta_1\left(ah\tau' + \frac{aM}{8} \bigg| 2\tau'\right).
\end{equation}

Notice that both $a$ and $h$ are odd.  We may therefore write $ah = 2N + 1$, so that

\begin{equation}
\vartheta_1\left(v| 2\tau'\right) =  (-1)^N \exp(-\pi i N (2\tau' + aM/4 + 2\tau'N))\vartheta_1\left(\tau' + \frac{aM}{4} \bigg| 2\tau'\right)\label{theta2s2}.
\end{equation}

We now shift from $\vartheta_1$ to $\vartheta_4$.  Write

\begin{align}
\vartheta_1\left(\tau' + \frac{aM}{8} \bigg| 2\tau'\right) &= i \exp\left(\frac{-\pi i}{8}(12\tau' + aM/8)\right) \vartheta_4\left(\frac{aM}{8} \bigg| 2\tau'\right)\label{theta2s3}.
\end{align}

We express $\vartheta_4$ as an infinite product:

\begin{equation}
\vartheta_4\left(\frac{aM}{8} \bigg| 2\tau'\right) = (y^2;y^2)_{\infty}(\rho_{a,2} y;y^2)_{\infty} (\rho_{a,2}^{-1}y;y^2)_{\infty}\label{theta2s4}.
\end{equation}  Combining (\ref{etas2}), (\ref{theta2s1}), (\ref{theta2s2}), (\ref{theta2s3}), (\ref{theta2s4}), and simplifying, we have

\begin{equation}
F_a(q) = \frac{1}{\sqrt{2}} \omega_{a,2} (h,k) \exp\left(\frac{\pi}{8k} \left( z(4a - 5)\right)\right) \Psi_{a,2} (y)\label{gen2final},
\end{equation} where

\begin{equation}
\Psi_{a,2}(q) = \frac{(q^2;q^2)_{\infty}^2}{(q^4;q^4)_{\infty}f(-\rho_{a,2} q;-\rho_{a,2}^{-1}q)}\label{psi2},
\end{equation} and

\begin{equation}
\omega_{a,2} (h,k) = \frac{i(-1)^{\lfloor \frac{ah}{2} \rfloor}}{\epsilon(2,8)\epsilon(2,4)}\exp\left( \frac{\pi i}{4k}\left( 1 - (4a - 3)(2hH_2 + 1) \right) \right)\label{rootun2}.
\end{equation}

\subsection{$GCD(k,8) = 1$}

With $d=4$, we apply (\ref{matrix4c8}) to $2\tau'$, and (\ref{matrix4c4}) to $4\tau'$, so that we have

\begin{equation}
\frac{\eta(8\tau)^2}{\eta(4\tau)} = \frac{1}{4z^{1/2}} \frac{\epsilon(1,8)^2}{\epsilon(1,4)} \frac{\eta(\tau')^2}{\eta(2\tau')}\label{etas1}.
\end{equation}

Returning to $\vartheta_1$,

\begin{equation}
\vartheta_1(a\tau|8\tau) = -i \epsilon(1,8)^3 \frac{1}{2\sqrt{2}z^{1/2}} e^{8\pi k z v^2} \vartheta_1\left( v | \tau' \right)\label{theta1s1}.
\end{equation}  And

\begin{equation}
\vartheta_1\left( v | \tau' \right) = \vartheta_1\left( ah\tau' + \frac{aM}{8} \bigg| \tau' \right)\label{theta1s2}.
\end{equation}

Recognizing that we may extract $ah\tau'$ altogether from our first variable, and recognizing that $(-1)^{ah} = (-1)^h$, we have

\begin{equation}
\vartheta_1\left(v| \tau' \right) =  (-1)^{h} \exp(-\pi i ah (ah\tau' + aM/4))\vartheta_1\left(\frac{aM}{8} \bigg| \tau'\right)\label{theta1s3}.
\end{equation}

We may now write $\vartheta_1\left(\frac{aM}{8} \bigg| \tau'\right)$ in its classic product form \cite[Chapter 10]{Rademacher3}:

\begin{align}
\vartheta_1\left(\frac{aM}{8}\bigg|\tau'\right) &= 2 e^{\pi i \tau'/4}\sin(\pi aM/8)\notag \\ \times \prod_{m=1}^{\infty}&(1-e^{2\pi i m \tau'})(1-e^{2\pi i m \tau' + 2\pi i aM/8}) (1-e^{2\pi i m \tau' - 2\pi i aM/8})\label{theta1s4}\\ &= 2 e^{\pi i \tau'/4}\sin(\pi aM/8)(y;y)_{\infty}(\rho_{a,1} y;y)_{\infty}(\rho_{a,1}^{-1}y;y)_{\infty}\label{theta1s5}.
\end{align}

Examining the sine function, let $aM = 8N + c$, with $c$ the least positive residue of $aM \pmod 8$.  Then

\begin{equation}
\sin\left(\frac{\pi aM}{8}\right) = (-1)^{N} \sin\left(\frac{\pi c}{8}\right)\label{sine1}.
\end{equation}

Notice that $\sin\left(\frac{\pi c}{8}\right) > 0$.  We know that since

\begin{equation}
M = -\frac{1}{k}(8hH_1 + 1),
\end{equation} and since $(k,8)=1$, therefore

\begin{equation}
c \equiv -ak^{-1} \pmod 8.
\end{equation}  Moreover, $k$ is odd, so $k^{-1}\equiv k \pmod 8$.  So

\begin{equation}
\sin\left(\frac{\pi c}{8}\right) = \left| \sin\left( \frac{\pi ak}{8} \right) \right|.
\end{equation}

Combining (\ref{etas1}), (\ref{theta1s1}), (\ref{theta1s2}), (\ref{theta1s3}), (\ref{theta1s4}), (\ref{theta1s5}), (\ref{sine1}), and simplifying, we have:

\begin{equation}
F_a(q) = \frac{1}{2\sqrt{2}} \omega_{a,1} (h,k) \left|\csc \left(\frac{\pi a k}{8}\right) \right| \exp\left(\frac{\pi}{8k} \left(\frac{1}{4z} + z(4a - 5)\right)\right) \Psi_{a,1} (y)\label{gen1final},
\end{equation} where

\begin{equation}
\Psi_{a,1}(y) = \frac{(y;y)_{\infty}}{(y^2;y^2)_{\infty}(\rho_{a,1} y;y)_{\infty}(\rho_{a,1}^{-1}y;y)_{\infty}}\label{psi1},
\end{equation} and

\begin{multline}
\omega_{a,1} (h,k) = \frac{(-1)^{\lfloor \frac{-a(8hH_1 + 1)}{8k} \rfloor + h - 1}}{\epsilon (1,8) \epsilon (1,4)}\\ \times \exp\left( \frac{\pi i}{4k} \left(4h (1 - a + hH_1(3 - 4a)) - H_1\right)\right)\label{rootun1}.
\end{multline}

\section{Integration}

Recall from Section 1 that

\begin{equation*}
g_a(n) = \frac{1}{2\pi i} \oint\limits_{\mathcal{C}} \frac{F_a(q)}{q^{n+1}} dq,
\end{equation*} while in Section 2 we described a contour for $\mathcal{C}$ that will prove useful for integration.  We will now begin the integration proper.

Let $N$ be some large positive integer, and let the corresponding Rademacher curve $P(N)$ be given.  Then we have the following:

\begin{equation}
g_a(n) = \frac{1}{2\pi i} \oint\limits_{\mathcal{C}} \frac{F_a(q)}{q^{n+1}} dq = \sum\limits_{k = 1}^{N} \sum_{\substack{0 \le h < k, 
\\ (h,k) = 1}} \frac{1}{2\pi i} \int\limits_{\gamma (h,k)} \frac{F_a(q)}{q^{n+1}} dq.
\end{equation}

In Section 3, we gave transformation equations for $F_a(q)$ depending on the divisibility properties of $k$.  We now separate our integral into the corresponding cases:

\begin{equation}
g_a(n) = g_a^{(8)} (n) + g_a^{(4)} (n) + g_a^{(2)} (n) + g_a^{(1)} (n),
\end{equation} with

\begin{equation}
g_a^{(d)}(n) = \sum_{\substack{(k,8) = d, \\ k \le N}} \sum_{\substack{0 \le h < k, \\ (h,k) = 1}} \frac{1}{2\pi i} \int\limits_{\gamma (h,k)} \frac{F_a(q)}{q^{n+1}} dq.
\end{equation}

In each case, we will tranform $F_a(q)$ by the following:

\begin{align}
g_a^{(d)}(n) =&  \sum_{\substack{(k,8)=d \\ k \le N}}\frac{i}{k} \sum_{\substack{0 \le h < k, \\ (h,k) = 1}} e^{-2\pi i n h/k}\notag \\ &\times \int\limits_{z_I(h,k)}^{z_T(h,k)} F_a\left(\exp\left(2\pi i \left(\frac{h}{k} + \frac{iz}{k}\right)\right)\right)e^{2\pi n z/k} dz\\ =& 2^{(\alpha - 3)/2} \sum_{\substack{(k,8)=d \\ k \le N}}\frac{i}{k} T_{a,d}(k) \sum_{\substack{0 \le h < k, \\ (h,k) = 1}}\omega_{a,d}(h,k) e^{-2\pi i n h/k}\notag \\ &\times \int\limits_{z_I(h,k)}^{z_T(h,k)} \exp\left(\frac{\pi}{8k}\left(\frac{\Lambda(a,d)}{z} + z(16n + 4a - 5)\right)\right) \Psi_{a,d}(y) dz,
\end{align} with $\omega_{a,d}(h,k)$ defined by (\ref{rootun8}), (\ref{rootun4}), (\ref{rootun2}), (\ref{rootun1}), $\Psi_{a,d}(y)=\sum_{j=0}^{\infty}\psi_{a,d}(j)y^j$ defined as $F_b(y)$ for $d=8$, and (\ref{psi4}), (\ref{psi2}), (\ref{psi1}), otherwise (note that $\psi_{a,d}(0)=1$);

\[
 \Lambda(a,d) = 
  \begin{cases} 
   (b-4)^2-8 & \text{if } d = 8 \\
   1       & \text{if } d = 4 \\
   0       & \text{if } d = 2 \\
   1/4    & \text{if } d = 1.
  \end{cases}
\]

\begin{equation}
\alpha = \log_2(d),
\end{equation} and

\[
 T_{a,d}(k) = 
  \begin{cases} 
   |\csc(\pi a k/8)| & \text{if } d = 1 \\
   1       & \text{otherwise}.
  \end{cases}
\]

Thereafter,

\begin{multline}
g_a^{(d)} (n) = 2^{(\alpha - 3)/2} \sum_{\substack{(k,8)=d \\ k \le N}} \frac{i}{k} T_{a,d}(k) \sum_{\substack{0 \le h < k, \\ (h,k) = 1}} \omega_{a,d}(h,k) e^{-2\pi i n h/k}\\ \times \left(I_{a,d}^{(1)}(h,k) + I_{a,d}^{(0)}(h,k)\right),
\end{multline} where

\begin{equation}
I_{a,d}^{(1)}(h,k) = \int\limits_{z_I(h,k)}^{z_T(h,k)} \exp\left( \frac{\pi}{8k} \left( \frac{\Lambda(a,d)}{z} + z(16n+4a-5) \right) \right) dz,
\end{equation} and

\begin{equation}
I_{a,d}^{(0)} (h,k) = \int\limits_{z_I(h,k)}^{z_T(h,k)} \exp\left( \frac{\pi}{8k} \left( \frac{\Lambda(a,d)}{z} + z(16n+4a-5) \right) \right) \sum\limits_{j=1}^{\infty} \psi_{a,d}(j)y^j dz.
\end{equation}

In each of our cases, we will show that $I_{a,d}^{(0)} (h,k)$ will contribute nothing to our final formula.

\begin{lemma}
For $d=8,4,2,1$,

\begin{equation}
\left|\sum_{\substack{0 \le h < k, \\ (h,k) = 1}}\omega_{a,d}(h,k)e^{-2\pi i n h/k}\right| = O\left( k^{2/3 + \epsilon} n^{1/3} \right).
\end{equation}
\end{lemma}

This result can be shown through Kloosterman sum estimation, using the techniques of Sali\'e \cite{Salie}.

\begin{lemma}

\begin{equation}
\left| I_{a,d}^{(0)}(h,k) \right| = O\left(\exp(3n\pi) N^{-1} \right).
\end{equation}

\end{lemma}

\begin{proof}

We may interchange the summation with the integration.  Also, remembering that $y = \exp\left(2\pi i \left(\frac{H_d}{k} + \frac{diz^{-1}}{8k}\right)\right)$,

\begin{align}
&I_{a,d}^{(0)} (h,k)\notag \\ =& \sum\limits_{j=1}^{\infty} \psi_{a,d}(j) e^{2\pi i H_dj/k}\notag\\ &\times \int\limits_{z_I(h,k)}^{z_T(h,k)} \exp\left( \frac{\pi}{8k} \left( \frac{\Lambda(a,d)}{z} + z(16n+4a-5) \right) \right) e^{-2dj\pi z^{-1}/8k} dz \\ =&  \sum\limits_{j=1}^{\infty} \psi_{a,d}(j) e^{2\pi i H_dj/k}\notag\\ &\times \int\limits_{z_I(h,k)}^{z_T(h,k)} \exp\left( \frac{\pi}{8k} \left( \frac{\Lambda(a,d)-2dj}{z} + z(16n+4a-5) \right) \right) dz.
\end{align}

Notice that no matter the permitted value of $d$, the coefficient of $1/z$ in the exponent of the integrand is now always negative.

Taking advantage of the fact that on and within $K_k^{(-)}$, $\Re(1/z) \ge k$ and $\Re(z) \le 1/k$, we now examine the magnitude of the integrand:

\begin{align}
&\left|\exp\left( \frac{\pi}{8k} \left( \frac{\Lambda(a,d)-2dj}{z} + z(16n+4a-5) \right) \right)\right| \\ =& \exp\left( \frac{\pi}{8k}(\Lambda(a,d)-2dj)\Re(1/z) + \frac{\pi}{8k}(16n+4a-5)\Re(z) \right) \\ \le& \exp\left( \frac{\pi(1-2dj)}{8} + \frac{\pi (16n+4a-5)}{8k^2} \right) \\ \le& \exp(-\pi j/8 + 3n\pi).
\end{align}

We therefore have

\begin{align}
&|I_{a,d}^{(0)}(h,k)|\notag \\ \le& \sum\limits_{j=1}^{\infty} |\psi_{a,d}(j)| |e^{2\pi i H_dj/k}|\notag\\ &\times \int\limits_{z_I(h,k)}^{z_T(h,k)} \left|\exp\left( \frac{\pi}{8k} \left( \frac{\Lambda(a,d)-2dj}{z} + z(16n+4a-5) \right) \right)\right| dz \\ \le&  \sum\limits_{j=1}^{\infty} |\psi_{a,d}(j)| \exp(-\pi j/8 + 3n\pi) \int\limits_{z_I(h,k)}^{z_T(h,k)} dz.
\end{align}

Recall that we are integrating along the circle $K_k^{(-)}$ in the $z$-plane.  We will now now deform our contour so that it is a chord connecting $z_I$ and $z_T$ along $K_k^{(-)}$.  

\begin{figure}[h]
\centering
\begin{tikzpicture}
    \begin{scope}[thick,font=\scriptsize][set layers]
    \draw [->] (-4,0) -- (4,0) node [above left]  {$\Re\{z\}$};
    \draw [->] (-3,-4) -- (-3,4) node [below left] {$\Im\{z\}$};
\draw [->] [thick] (-1.18,2.74) -- (-2.1,-2.15);
\draw [-] (0,-0.1) -- (0,0.1) node [above] {$1/2k$};
    \end{scope}
    \draw[very thin] (0,0) circle (3);
    \node [below right,black] at (3,2) {$K_k^{(-)}: \left|z - \frac{1}{2k}\right| = \frac{1}{2k}$};
\node [above,black] at (-2.1,2.5) {$z_I(h,k)$};
\node [below,black] at (-2.1,-2.6) {$z_T(h,k)$};
\fill[black] (-1.18,2.74) circle (2pt);
\fill[black] (-2.1,-2.15) circle (2pt);
\end{tikzpicture}
\caption{$K_k^{(-)}$ with the chord connecting $z_I(h,k)$ to $z_T(h,k)$.}\label{circlek2}
\end{figure}
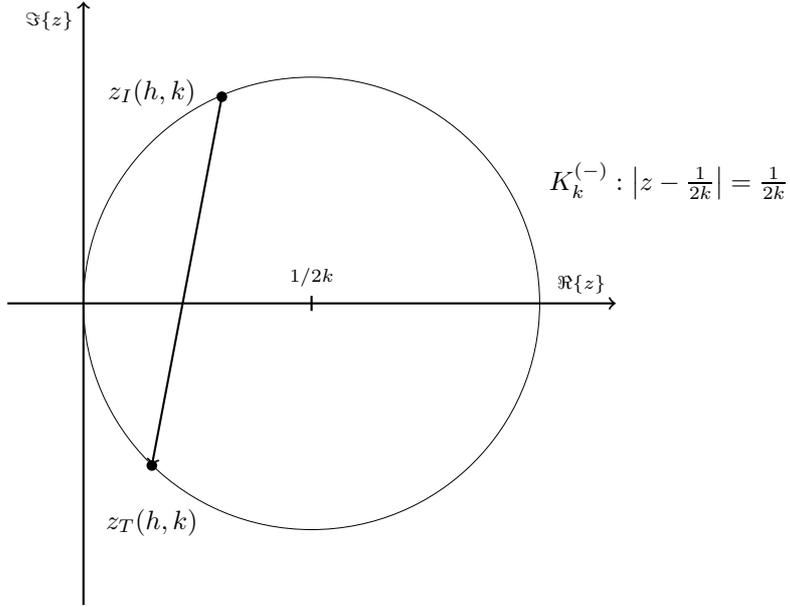

Recognizing from Lemma 1 that the length of such a chord is bounded above by a constant multiple of $N^{-1}$, we have

\begin{align}
|I_{a,d}^{(0)}(h,k)| &= O\left(\sum\limits_{j=1}^{\infty} |\psi_{a,d}(j)| \exp(-\pi j/8 + 3n\pi) N^{-1}\right) \\ &= O\left(\exp(3n\pi)N^{-1}\sum\limits_{j=1}^{\infty} |\psi_{a,d}(j)| \exp(-\pi j)\right) \\ &= O(\exp(3n\pi)N^{-1}).
\end{align}

\end{proof}

\begin{lemma}

Let $\epsilon > 0$.  Then

\begin{multline}
\left| 2^{(\alpha - 3)/2} \sum_{\substack{(k,8)=d \\ k \le N}} \frac{i}{k} T_{a,d}(k) \sum_{\substack{0 \le h < k, \\ (h,k) = 1}} \omega_{a,d}(h,k) e^{-2\pi i n h/k} I_{a,d}^{(0)}(h,k)\right|\\ = O\big( e^{3n\pi} n^{1/3} N^{-1/3 + \epsilon} \big).
\end{multline}

\end{lemma}

\begin{proof}

We note that since $2^{(\alpha - 3)/2}$ and $T_{a,d}(k)$ are bounded, we may disregard both in our estimation.  We now take the previous result into account:

\begin{align}
&\left| \sum_{\substack{(k,8)=d \\ k \le N}} \frac{i}{k} \sum_{\substack{0 \le h < k, \\ (h,k) = 1}} \omega_{a,d}(h,k) e^{-2\pi i n h/k} I_{a,d}^{(0)}(h,k)\right|\notag \\ &\le \sum_{\substack{(k,8)=d \\ k \le N}} \frac{1}{kN} e^{3n\pi} \left| \sum_{\substack{0 \le h < k, \\ (h,k) = 1}} \omega_{a,d}(h,k) e^{-2\pi i n h/k} \right|. 
\end{align}

With Lemma 3, we know that

\begin{equation}
\left| \sum_{\substack{0 \le h < k, \\ (h,k) = 1}} \omega_{a,d}(h,k) e^{-2\pi i n h/k} \right| = O\left(k^{2/3 + \epsilon}n^{1/3}\right).
\end{equation}

This gives us

\begin{align}
&\left| \sum_{\substack{(k,8)=d \\ k \le N}} \frac{i}{k} \sum_{\substack{0 \le h < k, \\ (h,k) = 1}} \omega_{a,d}(h,k) e^{-2\pi i n h/k} I_{a,d}^{(0)}(h,k)\right|\notag\\ =& O \left( \left|\sum_{\substack{(k,8)=d \\ k \le N}} \frac{1}{kN} e^{3n\pi} k^{2/3 + \epsilon}n^{1/3} \right| \right) = O \left( \left| e^{3n\pi}n^{1/3}N^{-1} \sum_{k = 1}^{N} \frac{ k^{2/3 + \epsilon}}{k} \right| \right).
\end{align}

Recognizing that

\begin{equation}
\sum_{k = 1}^{N} \frac{ k^{2/3 + \epsilon}}{k} = \sum_{k = 1}^{N} \frac{ k^{2/3 + 2\epsilon}}{k^{1+\epsilon}} \le \sum_{k = 1}^{N} \frac{ N^{2/3 + 2\epsilon}}{k^{1+\epsilon}} = N^{2/3 + 2\epsilon} \sum_{k = 1}^{N} \frac{1}{k^{1+\epsilon}},
\end{equation} that $\sum_{k = 1}^{N} \frac{1}{k^{1+\epsilon}}$ is bounded above as $N$ gets large, and finally noting that we may replace $2\epsilon$ with $\epsilon$, we now have

\begin{equation}
O\left( \left| e^{3n\pi}n^{1/3}N^{-1} \sum_{k = 1}^{N} \frac{ k^{2/3 + \epsilon}}{k} \right| \right) = O( e^{3n\pi}n^{1/3}N^{-1/3 + \epsilon}),
\end{equation} and the proof is completed.

\end{proof}

We now have 

\begin{multline}
g_a^{(d)}(n) = 2^{(\alpha - 3)/2} \sum_{\substack{(k,8)=d \\ k \le N}} \frac{i}{k} T_{a,d}(k) \sum_{\substack{0 \le h < k, \\ (h,k) = 1}} \omega_{a,d}(h,k) e^{-2\pi i n h/k} I_{a,d}^{(1)}(h,k)\\ + O\left( e^{3n\pi} n^{1/3} N^{-1/3 + \epsilon} \right).
\end{multline}

Our object now will be to put $I_{a,d}^{(1)}(h,k)$ into a form approachable from the theory of Bessel functions.

We now return to the original Rademacher contour of $I_{a,d}^{(1)}(h,k)$, along a portion of $K_k^{(-)}$.  The brilliance of the contour becomes clear once it is realized that $\Re(1/z) = k$, i.e. is a constant, provided we remain along $K_k^{(-)}$ (and avoid $z = 0$, of course).  We wish to make use of the whole of $K_k^{(-)}$, so we will make adjustments to the contour as follows:

\begin{multline}
I_{a,d}^{(1)}(h,k) = \left( \oint_{K_k^{(-)}} - \int\limits_0^{z_I(h,k)} - \int\limits_{z_T(h,k)}^0 \right)\\ \exp\left( \frac{\pi}{8k} \left( \frac{\Lambda(a,d)}{z} + z(16n+4a-5) \right) \right)dz\label{intcircle}.
\end{multline}

Notice that $\int\limits_0^{z_I(h,k)}$ and $\int\limits_{z_T(h,k)}^0$ are improper: the integrand is not defined at $z = 0$.  We interpret these integrals as limits in which a variable approaches $0$. We will now show that $\int\limits_0^{z_I(h,k)}$ and $\int\limits_{z_T(h,k)}^0$ will not contribute anything of importance:

\begin{lemma}

\begin{multline}
\left| \int\limits_{z_T(h,k)}^0 \exp\left( \frac{\pi}{8k} \left( \frac{1}{z} + z(16n+4a-5) \right) \right)dz \right|,\\ \left| \int\limits_0^{z_I(h,k)}\exp\left( \frac{\pi}{8k} \left( \frac{\Lambda(a,d)}{z} + z(16n+4a-5) \right) \right)dz \right| \\ = O\big(\exp(3n\pi) N^{-1} \big).
\end{multline}

\end{lemma}

\begin{proof}

We will keep on $K_k^{(-)}$ for these estimations.  Since the estimation is almost identical in either case, we will work with the integral $\int\limits_{z_T(h,k)}^0$.  We begin by estimating the integrand of the integral:

\begin{align}
&\left|\exp\left( \frac{\pi}{8k} \left( \frac{\Lambda(a,d)}{z} + z(16n+4a-5) \right) \right)\right|\notag \\ &= \exp\left( \frac{\pi}{8k} \left( \Re(1/z) + \Re(z)(16n+4a-5) \right) \right) \\ &\le \exp\left( \frac{\pi}{8k} \left( k + \frac{16n+4a-5}{k} \right) \right) \\ &\le \exp\bigg( \frac{\pi}{8} + \frac{\pi (16n + 4a - 5)}{8k^2} \bigg) \\ &\le \exp(3n\pi ).
\end{align}

We now estimate the path of integration:

The chord connecting 0 with $z_T(h,k)$ can be no longer than the diameter of $K^{(-)}$, so the length along the arc from 0 to $z_T(h,k)$ can be no longer than $|z_T(h,k)|\frac{\pi}{2}$.  Since $|z_T(h,k)| =O\left(N^{-1}\right)$, we have a path length that is $O(N^{-1})$.  This gives us

\begin{align}
&\left| \int\limits_{z_T(h,k)}^0\exp\left( \frac{\pi}{8k} \left( \frac{\Lambda(a,d)}{z} + z(16n+4a-5) \right) \right)dz \right|\notag \\ &\le \int\limits_{z_T(h,k)}^0 \left| \exp\left( \frac{\pi}{8k} \left( \frac{\Lambda(a,d)}{z} + z(16n+4a-5) \right) \right) \right| dz \\ &\le \exp(3n\pi ) \int\limits_{z_T(h,k)}^0 dz \\ &= O\left(\exp(3n\pi )N^{-1}\right).
\end{align}

The case for $\int\limits_0^{z_I(h,k)}$ is almost identical.

\end{proof}

As a consequence of the previous Lemmas 2, 3, 4, and 5, we have

\begin{theorem}

\begin{multline}
g_a^{(d)} (n) = 2^{(\alpha - 3)/2} \sum_{\substack{(k,8)=d, \\ k \le N}} \frac{i}{k} T_{a,d}(k) \sum_{\substack{0 \le h < k, \\ (h,k) = 1}} \omega_{a,d}(h,k) e^{-2\pi i n h/k}\\ \times \oint_{K_k^{(-)}} \exp\left( \frac{\pi}{8k} \left( \frac{\Lambda(a,d)}{z} + z(16n+4a-5) \right) \right)dz\\ + O\left( e^{3n\pi} n^{1/3} N^{-1/3 + \epsilon} \right) \label{gcd8}.
\end{multline}

\end{theorem}

\subsection{Estimating $g_a^{(8)}(n)$}

In the case of $d=8$ we can discard a large portion of what remains.  Notice that

\begin{equation}
I_{a,8}^{(1)}(h,k) = \int\limits_{z_I(h,k)}^{z_T(h,k)} \exp\left( \frac{\pi}{8k} \left( \frac{(b-4)^2 - 8}{z} + z(16n+4a-5) \right) \right) dz,
\end{equation} with $b \equiv ah \pmod 8$.  If $b = 1,7$, then the coefficient of $1/z$ in the exponent is 1.  However, if $b=3,5$, then the coefficient is $-7$, and by almost identical reasoning of Lemmas 2, 3, 4, applied to $I_{a,8}^{(1)}(h,k)$, we have

\begin{equation}
\left| I_{a,8}^{(1)}(h,k) \right| = O\left(\exp(3n\pi) N^{-1} \right).
\end{equation}  Now $\alpha = 3$ and $T_{a,d}(k)=1$.  Since $b=1,7$ implies $h \equiv \pm a \pmod 8$, we have

\begin{multline}
g_a^{(8)}(n) =  \sum_{\substack{(k,8)=8 \\ k \le N}} \frac{i}{k} \sum_{\substack{0 \le h < k, \\ (h,k) = 1 \\ h \equiv \pm a \pmod 8}} \omega_{a,8}(h,k) e^{-2\pi i n h/k}\\ \times \oint_{K_k^{(-)}} \exp\left( \frac{\pi}{8k} \left( \frac{1}{z} + z(16n+4a-5) \right) \right)dz + O\left( e^{3n\pi} n^{1/3} N^{-1/3 + \epsilon} \right)\label{gcd8real}.
\end{multline}

\subsection{Estimating $g_a^{(4)}(n)$}

Lemmas 2, 3, 4, and 5 are sufficient to complete the estimation of $g_a^{(4)}(n)$.  With $\alpha = 2$ and $T_{a,d}(k)=1$, we have

\begin{theorem}

\begin{multline}
g_a^{(4)} (n) = \frac{1}{\sqrt{2}} \sum_{\substack{(k,8)=4, \\ k \le N}} \frac{i}{k} \sum_{\substack{0 \le h < k, \\ (h,k) = 1}} \omega_{a,4}(h,k) e^{-2\pi i n h/k}\\ \times\oint_{K_k^{(-)}} \exp\left( \frac{\pi}{8k} \left( \frac{1}{z} + z(16n+4a-5) \right) \right)dz\\ + O\left( e^{3n\pi} n^{1/3} N^{-1/3 + \epsilon} \right)\label{gcd4}.
\end{multline}

\end{theorem}

\subsection{Estimating $g_a^{(2)}(n)$}

In the case of $d=2$, the coefficient of $1/z$ in the exponential of the integrand is never positive.  Therefore, we may immediately apply the reasoning of Lemmas 2, 3, 4 to both $I_{a,2}^{(0)}(h,k)$ and $I_{a,2}^{(1)}(h,k)$:

\begin{equation}
|I_{a,2}^{(0)}(h,k)| = |I_{a,2}^{(1)}(h,k)| = O\left(\exp(3n\pi)N^{-1}\right).
\end{equation}

Therefore,

\begin{theorem}

\begin{equation}
g_a^{(2)} (n) = O\left( e^{3n\pi} n^{1/3} N^{-1/3 + \epsilon} \right)\label{gcd2}.
\end{equation}

\end{theorem}

\subsection{Estimating $g_a^{(1)}(n)$}

Lemmas 2, 3, 4, and 5 are sufficient to complete the estimation of $g_a^{(1)}(n)$.  Noting that $\alpha = 0$ and $T_{a,d}(k)=|\csc(\pi a k/8)|$, we have

\begin{theorem}
\begin{multline}
g_a^{(1)} (n) = \frac{1}{2\sqrt{2}}  \sum_{\substack{(k,8)=1 \\ k \le N}} \frac{i}{k} \left|\csc\left(\frac{\pi a k}{8}\right)\right| \sum_{\substack{0 \le h < k, \\ (h,k) = 1}} \omega_{a,1}(h,k) e^{-2\pi i n h/k}\\ \times\oint_{K_k^{(-)}} \exp\left( \frac{\pi}{8k} \left( \frac{1}{4z} + z(16n+4a-5) \right) \right)dz\\ + O\left( e^{3n\pi} n^{1/3} N^{-1/3 + \epsilon} \right)\label{gcd1}.
\end{multline}
\end{theorem}

\section{Complete Formula}

Combining (\ref{gcd8real}), (\ref{gcd4}), (\ref{gcd2}), (\ref{gcd1}), and collecting the error terms, we have:

\begin{multline}
g_a(n) = \frac{1}{2\sqrt{2}}  \sum_{\substack{(k,8)=1 \\ k \le N}} \frac{i}{k} \left|\csc\left(\frac{\pi a k}{8}\right)\right| A_{a,1}(n,k)\\ \times \oint_{K_k^{(-)}} \exp\left( \frac{\pi}{8k} \left( \frac{1}{4z} + z(16n+4a-5) \right) \right)dz\\ + \frac{1}{\sqrt{2}} \sum_{\substack{(k,8)=4 \\ k \le N}} \frac{i}{k} A_{a,4}(n,k) \oint_{K_k^{(-)}} \exp\left( \frac{\pi}{8k} \left( \frac{1}{z} + z(16n+4a-5) \right) \right)dz\\ + \sum_{\substack{(k,8)=8 \\ k \le N}} \frac{i}{k} A_{a,8}(n,k) \oint_{K_k^{(-)}} \exp\left( \frac{\pi}{8k} \left( \frac{1}{z} + z(16n+4a-5) \right) \right)dz \\  + O\left( e^{3n\pi} n^{1/3} N^{-1/3 + \epsilon} \right)\label{penultimateetc},
\end{multline} with

\begin{equation}
A_{a,d}(n,k) = \sum_{\substack{0 \le h < k, \\ (h,k) = 1, \\ h \equiv \pm a\pmod d}} \omega_{a,d}(h,k) e^{-2\pi i n h/k}\label{finalrootof1}.
\end{equation}

We represent the remaining integrals with modified Bessel functions \cite{Watson}:

\begin{lemma}

\begin{multline}
\oint_{K_k^{(-)}} \exp\left( \frac{\pi}{8k} \left( \frac{1}{z} + z(16n+4a-5) \right) \right)dz\\ = \frac{-2\pi i}{\sqrt{16n + 4a - 5}} I_1 \left( \frac{\pi\sqrt{16n+4a-5}}{4k} \right),
\end{multline}

\begin{multline}
\oint_{K_k^{(-)}} \exp\left( \frac{\pi}{8k} \left( \frac{1}{4z} + z(16n+4a-5) \right) \right)dz\\ = \frac{-\pi i}{\sqrt{16n + 4a - 5}} I_1 \left( \frac{\pi\sqrt{16n+4a-5}}{8k} \right).
\end{multline}

\end{lemma}

The first equality may be proved by changing variables, first by $z = 1/w$, and then by $w = 8kt/\pi$.  We may then represent the integral with the modified Bessel function $I_1$ \cite{Watson}:  The second equality may be similarly proved by changing variables by $z = 1/w$, and then by $w = 32kt/\pi$.

\subsection{Finishing the Limit Process}

We now take (\ref{penultimateetc}), with $A_{a,d}(n,k)$ defined by (\ref{finalrootof1}), substitute and simplify through Lemma 6, and let $N \rightarrow \infty$.  We now have our final formula.

\begin{theorem}

Let $g_a(n)$ be the number of type-$a$ G\"ollnitz--Gordon partitions of $n$, with $a=1$ or $3$.  Then

\begin{multline}
g_a(n) = \frac{\pi\sqrt{2}}{4\sqrt{16n+4a-5}} \sum\limits_{(k,8) = 1} \left|\csc\left(\frac{\pi a k}{8}\right)\right| \frac{A_{a,1}(n,k)}{k} I_1 \left( \frac{\pi\sqrt{16n+4a-5}}{8k} \right)\\ + \frac{\pi\sqrt{2}}{\sqrt{16n+4a-5}} \sum\limits_{(k,8) = 4} \frac{A_{a,4}(n,k)}{k} I_1 \left( \frac{\pi\sqrt{16n+4a-5}}{4k} \right)\\ + \frac{2\pi}{\sqrt{16n+4a-5}} \sum\limits_{(k,8) = 8} \frac{A_{a,8}(n,k)}{k} I_1 \left( \frac{\pi\sqrt{16n+4a-5}}{4k} \right)\label{completed}.
\end{multline}

\end{theorem}

\section{Numerical Tests}

Mathematica was used to test (\ref{completed}), for $k$ truncated.  See the tables below.  For the first 200 positive integers, our formula with $k \le 3\sqrt{n}$ gives the correct value with an absolute error less than $ 0.33$.  Since $g_a(n)\in\mathbb{Z}$, we need only round our formulaic value to the nearest integer to achieve the correct answer.

\begin{table}
\begin{center}
\scalebox{0.8}{
\begin{tabular}{ | l | l | l | p{5cm} |}
\hline
$n$ & $g_1(n)$ & Eqn. (\ref{completed}), $a=1$, $1\le k \le 3\sqrt{n}$ & Absolute Error of (\ref{completed}) \\ \hline
\hline
1 & 1 & 0.7784305652 & 0.2215694348 \\ \hline
2 & 1 & 0.7196351376 & 0.2803648624 \\ \hline
3 & 1 & 1.114485490 & 0.114485490 \\ \hline
4 & 2 & 1.890769460 & 0.109230540 \\ \hline
5 & 2 & 2.146945231 & 0.146945231 \\ \hline
6 & 2 & 2.174897898 & 0.174897898 \\ \hline
7 & 3 & 2.917027886 & 0.082972114 \\ \hline
8 & 4 & 3.994864237 & 0.005135763 \\ \hline
9 & 5 & 4.903833678 & 0.096166322 \\ \hline
10 & 5 & 5.108441112 & 0.108441112 \\ \hline
20 & 26 & 26.07125673 & 0.07125673 \\ \hline
40 & 288 & 287.9388309 & 0.0611691 \\ \hline
60 & 1989 & 1988.942843 & 0.057157 \\ \hline
80 & 10570 & 10569.99993 & 0.00007 \\ \hline
100 & 47091 & 47090.99132 & 0.00868 \\ \hline
150 & 1191854 & 1191853.996 & 0.004 \\ \hline
200 & 18900623 & 18900622.99 & 0.001 \\ \hline
\hline
\end{tabular}}
\caption{$g_1(n)$ compared to (\ref{completed}) truncated for $k$, with $a=1$.}\label{title1a}
\end{center}
\end{table}

\begin{table}
\begin{center}
\scalebox{0.8}{
\begin{tabular}{ | l | l | l | p{5cm} |}
\hline
$n$ & $g_3(n)$ & Eqn. (\ref{completed}), $a=3$, $1\le k \le 3\sqrt{n}$ & Absolute Error of (\ref{completed}) \\ \hline
\hline
1 & 0 & 0.2908871603 & 0.2908871603 \\ \hline
2 & 0 & 0.1385488254 & 0.1385488254 \\ \hline
3 & 1 & 0.8129880460 & 0.1870119540 \\ \hline
4 & 1 & 0.9584818018 & 0.0415181982 \\ \hline
5 & 1 & 0.8666320258 & 0.1333679742 \\ \hline
6 & 1 & 0.9177374697 & 0.0822625303 \\ \hline
7 & 1 & 1.323340028 & 0.323340028 \\ \hline
8 & 2 & 2.095679009 & 0.095679009 \\ \hline
9 & 2 & 2.042654099 & 0.042654099 \\ \hline
10 & 2 & 1.953812941 & 0.046187059 \\ \hline
20 & 12 & 12.01649403 & 0.01649403 \\ \hline
40 & 127 & 126.9760443 & 0.0239557 \\ \hline
60 & 865 & 865.0090307 & 0.0090307 \\ \hline
80 & 4560 & 4560.002784 & 0.002784 \\ \hline
100 & 20223 & 20223.00416 & 0.00416 \\ \hline
150 & 508454 & 508454.0481 & 0.0481 \\ \hline
200 & 8034534 & 8034534.006 & 0.006 \\ \hline
\hline
\end{tabular}}
\caption{$g_3(n)$ compared to (\ref{completed}) truncated for $k$, with $a=3$.}\label{title2a}
\end{center}
\end{table}

\end{document}